\documentclass[11pt]{amsart}

\usepackage{amsmath}
\usepackage{amssymb}
\usepackage{amsthm}

\usepackage[margin=1.5in]{geometry} 
\usepackage{amsmath,amsthm,amssymb, bbm, dsfont}
\usepackage{hyperref}
\usepackage{adjustbox}
\usepackage[table,xcdraw]{xcolor}
\usepackage{listings}
\usepackage{amsmath}
\usepackage{amssymb}
\usepackage{bm}
\usepackage{graphicx}
\usepackage{psfrag}
\usepackage{color}
\usepackage{hyperref}
\hypersetup{colorlinks=true, linkcolor=blue, citecolor=magenta, urlcolor=wine}
\usepackage{url}
\usepackage{algpseudocode}
\usepackage{fancyhdr}
\usepackage{mathtools}
\usepackage{tikz-cd}
\usepackage{xy}
\input xy
\xyoption{all}
\usepackage{calrsfs}

\newtheorem{theorem}{Theorem}[section]
\newtheorem{lemma}[theorem]{Lemma}
\newtheorem{remark}[theorem]{Remark}
\numberwithin{equation}{section}

\def\Z{\mathbb{Z}}
\def\Q{\mathbb{Q}}
\usepackage[colorinlistoftodos]{todonotes}
\DeclareMathOperator{\rad}{rad}
\DeclareMathOperator{\Cl}{Cl}

\makeatletter
\def\pmod#1{\allowbreak\mkern5mu({\operator@font mod}\,\,#1)}

\makeatother

\keywords{pure fields, real Kummer extensions, genus field, genus number, counting fields, arithmetic statistics.}

\subjclass[2020]{Primary: 11R45, 11N45; Secondary: 11R29}

\begin{document}

\title{The average genus number for pure fields of prime degree}

\author{Sambhabi Bose}
\address{Department of Mathematics\\University of California, Berkeley\\Berkeley, CA 94720, US}
\email{sbose812@berkeley.edu}

\author{Kevin J. McGown}
\address{Department of Mathematics and Statistics\\California State University, Chico\\Chico, CA 95929, US}
\email{kmcgown@csuchico.edu}

\author{Ishan Panpaliya}
\address{Department of Mathematics\\Seattle University\\Seattle, WA 98122, US}
\email{ipanpaliya@seattleu.edu}

\author{Natalie Welling}
\address{Department of Mathematics\\University of Massachusetts Amherst\\Amherst, MA 01003, US}
\email{nwelling@umass.edu}

\author{Laney Williams}
\address{Department of Mathematics and Statistics\\California State University, Chico\\Chico, CA 95929, US}
\email{lcwilliams1@csuchico.edu}

\date{\today}

\begin{abstract}
Let $\ell\geq 5$ be prime.
Let $\mathcal{F}_\ell$ be the collection of (isomorphism classes of) pure number fields $\Q(\sqrt[\ell]{a})$ of degree $\ell$, ordered by the absolute value of their discriminant.  In 2018, Benli proved a counting theorem for $\mathcal{F}_\ell$, generalizing a previous theorem of Cohen and Morra when $\ell=3$. We prove that the proportion of pure fields of degree $\ell$ with genus number one is asymptotic to $(A_\ell \log X)^{-1}$ and that the average genus number for pure fields of degree $\ell$ is asymptotic to $B_\ell(\log X)^{\ell-1}$. 
Both $A_\ell$ and $B_\ell$ are expressed explicitly as a product over primes.
\end{abstract}

\bigskip
%%%%%%%%%%%%%%%%
\maketitle

\section{Introduction}

The subject of counting number fields, and more generally arithmetic statistics, has been one of intense interest over the last few decades.
Given a collection of number fields $\mathcal{F}$ and a
function $f \colon \mathcal{F}\to\mathbb{R}_{\geq 0}$, one can ask for an asymptotic (or the limiting value) for the 
expression
$$
  \frac{\sum_{\substack{K\in\mathcal{F}\\|\Delta_K|\leq X}}f(K)}{\sum_{\substack{K\in\mathcal{F}\\|\Delta_K|\leq X}}1}
  \,.
$$
Here our fields $K$ are ordered by the absolute value of the discriminant $\Delta_K$,
but alternative height functions could be considered.
Examples of functions $f(K)$ of very high interest are
the class number $h_K$, the indicator function for $h_K=1$,
or the size of the $p$-torsion of the class group $\#(\Cl_K[p])$.
A prototypical example is the famous Davenport--Heilbronn theorems
(see~\cite{MR491593,MR3090184,MR3127806,MR4768704}) which
allow one to obtain the average of $\#(\Cl_K[3])$ over the collection of quadratic fields of fixed signature.
However, many such questions are completely out of reach.
Recently, McGown--Tucker, Kim, and Frei--Loughran--Newton studied the statistics of the
genus number $g_K$ (see~\cite{MR4613609,MR4649640,MR4048607,MR4598183}).  This quantity is supported at the bad primes,
in the language of the Cohen--Lenstra heuristics (see~\cite{MR756082}),
and is much more amenable to study than $h_K$.

In light of Malle's conjecture (see~\cite{MR1884706,MR2068887}), it is most common to take $\mathcal{F}$ to consist of all
degree $d$ extensions of $\Q$ having a specified Galois closure $G$.  In this case,
Malle's conjecture predicts that the count of such extensions with $|\Delta_K|\leq X$
is asymptotic to $CX^{\alpha}(\log X)^\beta$
for appropriate choices of $0<\alpha\leq 1$, $\beta\geq 0$, and $C>0$.
(One can also consider the more general problem where $\Q$ is replaced by any base field $k$.)

On the other hand, there are other interesting collections of fields that one might consider.
In this paper we will be concerned with the collection of pure fields of prime degree~$\ell$.
Let $\mathcal{F}_\ell$
denote the collection of all (isomorphism classes of) degree $\ell$ number fields 
of the form $K=\Q(\sqrt[\ell]{a})$ with $a\in\Z$.
Note that we must have that $a\in\Z$ is not an $\ell$-th power; moreover,
we may assume that $a$ is positive and $\ell$-free, which we will tacitly do throughout.
Define the counting function
$$
N_\ell(X)=\#\{K\in\mathcal{F}_\ell: |\Delta_K|\leq X\}
\,.
$$
Benli proves (see~\cite{MR3814003}) that for all $\ell \geq 5$, the asymptotic
\begin{equation}\label{E:Benli}
  N_\ell(X)\sim C_\ell X^{\frac{1}{\ell-1}}(\log X)^{\ell-2}
  \,
\end{equation}
holds for an explicit constant $C_\ell$.  Actually, she proves something stronger,
but at present we will only be concerned with asymptotics.
We note that Cohen and Morra had previously treated the case of $\ell=3$ (see~\cite{MR2745550}).
The aim of this paper is to determine
the proportion of fields with genus number one and
the average genus number for the family $\mathcal{F}_\ell$.
More precisely, we seek an asymptotic as $X$ tends to infinity for these quantities.

\begin{theorem}\label{T:one}
The proportion of pure fields of prime degree $\ell\geq 3$ with genus number one equals
$$
  \frac{\#\{K\in\mathcal{F}_\ell:g_K=1\,,\;|\Delta_K|\leq X\}}{\#\{K\in\mathcal{F}_\ell:|\Delta_K|\leq X\}}
  \sim
  \frac{1}{A_\ell \log X},
$$
where
$$
  A_\ell=\frac{1}{(\ell-1)(\ell-2)}
  \prod_p
  A_{\ell,p}
  \,,\quad
  A_{\ell,p}=
  \left(1-\frac{1}{p}\right)
  \begin{cases}
  1+\frac{\ell-1}{p}, & p\equiv 1\pmod{\ell}\,,\\
1, & p\not\equiv 1\pmod{\ell}\,.
  \end{cases}  
$$  
\end{theorem}

We remark that the $A_{\ell,p}$ can also be written as a single formula in terms of the
$\ell$-th power residue symbol, and the same is true for $B_{\ell,p}$ in the following theorem.

\begin{theorem}\label{T:average}
The average genus number over pure fields of prime degree $\ell\geq 3$ satisfies
$$
  \frac{\sum_{\substack{K\in\mathcal{F}_\ell\\|\Delta_K|\leq X}}g_K}{\sum_{\substack{K\in\mathcal{F}_\ell\\|\Delta_K|\leq X}}1}
  \sim
  B_\ell (\log X)^{\ell-1}
  \,,\quad
  B_\ell=\frac{(\ell-2)!}{(2\ell-3)!(\ell-1)^{\ell-1}}\prod_p B_{\ell,p}
  \,,
$$
where
$$
B_{\ell,p}=\left(1-\frac{1}{p}\right)^{\ell-1}
\begin{cases}
\left(1+\frac{\ell(\ell-1)}{p}\right)\left(1+\frac{\ell-1}{p}\right)^{-1}, & p\equiv 1\pmod{\ell}\,,\\
1, & p\not\equiv 1\pmod{\ell}\,.
\end{cases}
$$
\end{theorem}

In Section~\ref{S:genus.background} we will give the definition of the
genus number and summarize previous results concerning statistics of genus numbers.
In Section~\ref{S:prelim}
we state some preliminary lemmas that will be required in the proofs of our results.
In Section~\ref{S:counting} we give a brief proof of the asymptotic for $N_\ell(X)$
given in (\ref{E:Benli}). Although our argument differs from the one given in~\cite{MR3814003},
the main reason for presenting the proof in its entirety
is that this will serve as an outline of the proof strategy we will employ in later sections.
Finally, in Sections~\ref{S:one} and~\ref{S:average} we prove Theorems~\ref{T:one}
and~\ref{T:average}, respectively.

\section{Background on the genus number}\label{S:genus.background}

We adopt the notation from the previous section, including writing $\mathcal{F}_\ell$ 
for the collection of pure fields of prime degree $\ell$. Throughout the paper, $\ell$ will always denote an odd prime.

We write $h_K$ and $h_K^+$ to denote the class number and narrow class number, respectively.
Fr\"ohlich (see~\cite{MR113868}) defined the genus field of a number field $K$
to be the maximal extension $K^*$ of $K$ that is unramified at all finite primes and is a compositum of the form $Kk^*$ where $k^*$ is absolutely abelian. 
The genus number of $K$ is defined as $g_K=[K^*:K]$,
and by class field theory $g_K\mid h^+_K$.  Moreover,
since $h_K$ and $h_K^+$ differ by a power of $2$, it follows from
Lemma~\ref{L:Ishida} that for $K\in\mathcal{F}_\ell$,
one has $g_K\mid h_K$.  In fact, this holds whenever $K$ has
odd prime degree.

Here we will briefly indicate previous works that have considered the study of the genus number
of a number field from a statistical perspective.  This section can be skipped by those who want to proceed directly to the proofs.
In~\cite{MR4613609}, the family of all cubic fields was considered.
As $S_3$-cubic fields make up 100\% of all cubic fields, from the perspective of that paper
it was only necessary to treat the case of $S_3$-cubic fields.  However, interesting subfamilies
of cubic fields include the cyclic cubic fields and the pure cubic fields.
We note that the collection of pure cubic fields $\mathcal{F}_3$ is the subfamily 
of $S_3$-cubic fields whose Galois closure contains $\Q(\sqrt{-3})$.

In~\cite{MR4048607}, the cases of cyclic and dihedral fields of any prime degree were considered.
The dihedral case is more difficult and therefore asymptotics are conditional on conjectures for $p$-torsion in class groups.
In~\cite{MR4649640}, the family of all quintic fields was considered.
As in the cubic case, the $S_5$-quintic fields constitute 100\% of all quintic fields,
and hence that work only considered $S_5$-quintic extensions. As $\mathcal{F}_5$ (the collection of pure quintic fields) is a subcollection of those
quintic fields whose Galois closure is the Frobenius group of order $20$,
these fields were not treated in that work.
Indeed, one can show that the Galois closure of $K\in\mathcal{F}_\ell$
is isomorphic to a semidirect product $C_\ell\rtimes C_{\ell-1}$.
Finally, in~\cite{MR4598183}, the case of all abelian extensions was considered.
In that work, results over an arbitrary base field $k$ were obtained.

\section{Preliminaries}\label{S:prelim}

Given $n\in\Z^+$, we will write $\rad(n)$ to denote the radical of $n$, and $\omega(n)$ for the number of distinct prime divisors of $n$.  We will refer to $n$ as $\ell$-free if it is not divisible by an $\ell$-th power; that is, $n$ is $\ell$-free if $\nu_p(n)<\ell$ for all primes $p$,
where $\nu_p$ denotes the $p$-adic valuation.
  
The following gives a formula for the discriminant of the fields
in $\mathcal{F}_\ell$ (see~\cite{MR184933,MR1205901}).

\begin{lemma}\label{L:field1}
For $K=\Q(\sqrt[\ell]{a})\in\mathcal{F}_\ell$ with $\rad(a)=n$, we have
\begin{align*}
  |\Delta_K|=\begin{cases}
  \ell^{\ell-2}n^{\ell-1}, & a^{\ell-1}\equiv 1\pmod{\ell^2},\\
  \ell^\ell n^{\ell-1}, & a^{\ell-1}\not\equiv 1\pmod{\ell^2}.
  \end{cases}
\end{align*}
\end{lemma}

The following result allows one to parametrize $\mathcal{F}_\ell$.
We leave this as an exercise, but the reader may wish to consult a reference on Kummer theory such as~\cite{MR1878556}.

\begin{lemma}\label{L:field2}
Suppose $a,b\in\Z$ are $\ell$-free.  Then
$\Q(\sqrt[\ell]{a})$ and $\Q(\sqrt[\ell]{b})$ are isomorphic if and only if
$a=b^i$ in $\Q^\times/\left(\Q^\times\right)^\ell$.
Moreover,
the function
$$\{a\in\Z_{\geq 2}\mid a\text{ is $\ell$-free}\}\to\mathcal{F}_\ell$$ given by
$a\mapsto\Q(\sqrt[\ell]{a})$
is a surjective $(\ell-1)$-to-$1$ map.
\end{lemma}

The following formula for the genus number appears in~\cite{MR584554, MR625129}.  The reader may also
want to consult~\cite{MR435028}.

\begin{lemma}\label{L:Ishida}
Let $K=\Q(\sqrt[\ell]{m})\in\mathcal{F}_\ell$.  Then
$$
  g_K=\prod_{\substack{p\mid m\\p\neq \ell}} (\ell,p-1)=\ell^{\widehat{\omega}(m)}
  \,.
$$
We denote by $\widehat{\omega}(n)$ the number of $p\mid n$
such that $p\equiv 1\pmod{\ell}$.
\end{lemma}

The following result is standard (see, for example, Chapter 4 of~\cite{MR2378655}).

\begin{lemma}[Mertens' Second Theorem in Residue Classes]\label{L:Mertens}
  Suppose $S\subseteq(\Z/m\Z)^\times$.  As $X\to\infty$, we have that
  $$
    \sum_{\substack{p\leq X\\p\in S\pmod{m}}}\frac{1}{p}=\frac{|S|}{\phi(m)}\log\log X + M + O\left(\frac{1}{\log X}\right).
  $$
  Note that both the constant $M$ and the implicit constant in the error term depend on the set $S$
  (and hence the modulus $m$).
\end{lemma}

The following result is well-known (see~\cite{MR131389}).

\begin{lemma}[Wirsing's Theorem]\label{L:Wirsing}
    Let $f(n)\geq 0$ be a multiplicative function.  Assume that 
    there is a constant $\tau>0$ such that
    \begin{equation} \label{eq: wirsing condition}
        \sum_{p \leq X} f(p) \sim \tau\cdot\frac{X}{\log X}
    \end{equation}
    and constants $\gamma_1 > 0$ and $0 \leq \gamma_2 < 2$
    such that $f(p^k) \leq \gamma_1 \gamma_2^k$ for all $k\geq 2$.
   Then we have
    \[\sum_{n \leq X} f(n) \sim \frac{X}{\log(X)} \cdot \frac{e^{-\gamma \tau}}{\Gamma(\tau)} \prod_{p \leq X} \left(1 + \frac{f(p)}{p} + \frac{f(p^2)}{p^2} + \cdots \right)\,.\]
\end{lemma}

\begin{remark}
Our applications of Wirsing's Theorem will always be to a sum of the form $\sum_{n\leq X}\mu^2(n)f(n)$.
Note that this simplifies both the hypothesis and the conclusion of Lemma~\ref{L:Wirsing}.
In addition, it will be helpful to note that for $\tau>0$,
one has
$$
\prod_{p\leq X}\left(1+\frac{\tau}{p}\right)\sim e^{\tau\gamma}(\log X)^\tau\prod_p\left(1+\frac{\tau}{p}\right)\left(1-\frac{1}{p}\right)^\tau
\,.
$$
In the sequel, we will require the use of this type of manipulation, which follows from Mertens' Third Theorem.
\end{remark}

\section{Counting pure fields of prime degree $\ell$}\label{S:counting}

Suppose $\ell\geq 3$ is prime.
We begin with the equation
\begin{align}
  \label{E:1}
   (\ell-1)N_\ell(X)+1
   &= 
  \sum_{n\leq a(\ell)X^{\frac{1}{\ell-1}}}
    \hspace{-2ex}
  \mu^2(n)
  \hspace{-4ex}
  \sum_{\substack{
  \rad(a)=n\\
  \text{$a$ is $\ell$-free}\\
  a^{\ell-1}\equiv 1\pmod{\ell^2}
  }}
    \hspace{-2ex}
  1
  +
  \sum_{n\leq b(\ell)X^{\frac{1}{\ell-1}}}
  \hspace{-2ex}
  \mu^2(n)
   \hspace{-4ex}
  \sum_{\substack{
  \rad(a)=n\\
  \text{$a$ is $\ell$-free}\\
  a^{\ell-1}\not\equiv 1\pmod{\ell^2}
  }}
  1,
\end{align}
where $a(\ell)=\ell^{-1+\frac{1}{\ell-1}}$ and
$b(\ell)=\ell^{-1-\frac{1}{\ell-1}}$,
which follows from
Lemmas~\ref{L:field1} and~\ref{L:field2}.

Using the orthogonality of Dirichlet characters,
the first summand on the right-hand side of (\ref{E:1}) becomes
\begin{align}
  \nonumber
  &
  \frac{1}{\phi(\ell^2)}
  \sum_{\chi\text{ mod $\ell^2$}}
  \sum_{n\leq a(\ell)X^{\frac{1}{\ell-1}}}
  \mu^2(n)
  \sum_{\substack{
  \rad(a)=n\\
  \text{$a$ is $\ell$-free}
  }}
    \chi^{\ell-1}(a)
    \\
    \nonumber
    &=
    \frac{1}{\ell}
  \sum_{\substack{\chi\text{ mod $\ell^2$}\\\chi^\ell=\chi_0}}
  \sum_{n\leq a(\ell)X^{\frac{1}{\ell-1}}}
  \mu^2(n)
  \sum_{\substack{
  \rad(a)=n\\
  \text{$a$ is $\ell$-free}
  }}
    \chi(a)
    \,.
\end{align}
This motivates us to define the following multiplicative function
and the associated summatory function (over squarefree integers):
$$
  f_\chi(n)=
  \sum_{\substack{
  \rad(a)=n\\
  \text{$a$ is $\ell$-free}
  }}
    \chi(a)
    \,,\qquad
  F_\chi(Y) = \sum_{n\leq Y}\mu^2(n)f_\chi(n)
  \,.
$$
Without the character, we similarly have
$$
  f(n)=
  \sum_{\substack{
  \rad(a)=n\\
  \text{$a$ is $\ell$-free}
  }}
  1
  =
  (\ell-1)^{\omega(n)}
  \,,\qquad
  F(Y)=\sum_{n\leq Y}\mu^2(n)f(n)
  \,.
$$
%This all leads to an expression for $(\ell-1)N_\ell(X)+1$,
%which is the quantity inside the curly braces of (\ref{E:1}):
This all leads to the following expression:
\begin{equation}\label{E:BIG}
(\ell-1)N_\ell(X)+1
=
\frac{1}{\ell}\sum_{\substack{\chi\text{ mod $\ell^2$}\\\chi^\ell=\chi_0}}
  \left[
  F_\chi(a(\ell)X^{\frac{1}{\ell-1}})
  -
    F_\chi\left(
   b(\ell)X^{\frac{1}{\ell-1}}
  \right)
\right]
  +
  F\left(
  b(\ell)X^{\frac{1}{\ell-1}}
  \right)
  \,.
\end{equation}
To give an asymptotic for $N_\ell(X)$,
it suffices to give an asymptotic for $F_\chi(Y)$ when $\chi=\chi_0$,
including the case of $F(Y)$,
and to bound $F_\chi(Y)$ when $\chi\neq \chi_0$.

\vspace{1ex}
\underline{Case $\chi=\chi_0$}:
We clearly have
$\sum_{p\leq Y}(\ell-1)\sim (\ell-1)Y/\log Y$
by the Prime Number Theorem and therefore
Lemma~\ref{L:Wirsing} implies
\begin{equation}\label{E:nochi}
  F(Y)\sim D_\ell Y(\log Y)^{\ell-2}
  \,,\quad
  D_\ell
  =
  \frac{1}{(\ell-2)!}\prod_{p}\left(1+\frac{\ell-1}{p}\right)\left(1-\frac{1}{p}\right)^{\ell-1}
  \,.
\end{equation}
Similarly, we have
\begin{equation}\label{E:chi0}
F_{\chi_0}(Y)\sim c(\ell)D_\ell Y(\log Y)^{\ell-2}
\,,\quad
c(\ell)=\left(1+\frac{\ell-1}{\ell}\right)^{-1}=\frac{\ell}{2\ell-1}
\,.
\end{equation}
\vspace{1ex}

\underline{Case $\chi\neq \chi_0$}:
We use the bound
$$|F_\chi(Y)|\leq\sum_{n\leq Y}\mu^2(n)|f_\chi(n)|$$ 
and apply Lemma~\ref{L:Wirsing} to the latter sum to obtain
\begin{align*}
F_\chi(Y)
\ll
\frac{Y}{\log Y}
\prod_{\substack{p\leq Y\\p^{\ell-1}\not\equiv 1\pmod{\ell^2}}}\left(1+\frac{1}{p}\right)
\prod_{\substack{p\leq Y\\p^{\ell-1}\equiv 1\pmod{\ell^2}}}\left(1+\frac{\ell-1}{p}\right)
\,.
\end{align*}
Taking logarithms of the two products above, using the Taylor expansion at $x = 0$ for $\log(1+x)$, and applying Lemma~\ref{L:Mertens} gives
\begin{align*}
&
\sum_{\substack{p\leq Y\\p^{\ell-1}\not\equiv 1\pmod{\ell^2}}}
\frac{1}{p}
+
\sum_{\substack{p\leq Y\\p^{\ell-1}\equiv 1\pmod{\ell^2}}}
\frac{\ell-1}{p}
+
\sum_{p\leq Y}O\left(\frac{1}{p^2}\right)
\\[1ex]
&\qquad=
\left(\frac{(\ell-1)^2}{\phi(\ell^2)}\cdot 1\right)\log\log Y+
\left(\frac{(\ell-1)}{\phi(\ell^2)}\cdot(\ell-1)\right)\log\log Y +O(1)
\\[1ex]
&\qquad=
2\left(1-\frac{1}{\ell}\right)\log\log Y + O(1)
\,.
\end{align*}
Therefore
\begin{equation}\label{E:haschi}
  F_\chi(Y)\ll Y(\log Y)^{1-\frac{2}{\ell}}
  \,.
\end{equation}
Putting (\ref{E:nochi}), (\ref{E:chi0}), (\ref{E:haschi}) into (\ref{E:BIG}),
and using
\begin{equation}\label{L:logpower}
\log\left(a(\ell)X^{\frac{1}{\ell-1}}\right)^{\ell-2}\sim
\frac{1}{(\ell-1)^{\ell-2}}(\log X)^{\ell-2}
\,,
\end{equation}
we establish (\ref{E:Benli}) with
\begin{align}
  \label{E:const}
  C_\ell&=\frac{1}{(\ell-1)^{\ell-1}}\left(\frac{a(\ell)-b(\ell)}{\ell}c(\ell)+b(\ell)\right)D_\ell
  \\
  \label{E:const2}
  &=
  \frac{\ell^{-\frac{1}{\ell-1}}}{\ell!(\ell-1)^{\ell-2}}\left(1+\frac{\ell^{\frac{2}{\ell-1}}-1}{2\ell-1}\right)
    \prod_{p}\left(1+\frac{\ell-1}{p}\right)\left(1-\frac{1}{p}\right)^{\ell-1}
  \,.
\end{align}
We note that our formula for the constant (\ref{E:const2}) differs slightly from the one
in~\cite{MR3814003}, but we believe it to be correct and it matches the formula in~\cite{MR2745550}
when specialized to $\ell=3$.

\section{Pure fields with genus number one}\label{S:one}

\begin{proof}[Proof of Theorem~\ref{T:one}]
As before we write $K=\Q(\sqrt[\ell]{a})\in\mathcal{F}_\ell$ with $\rad(a)=n$.
Lemma~\ref{L:Ishida} tells us that $g_K=\ell^{\widehat{\omega}(n)}$.
In order to only count fields $K$ with genus number one,
we impose the restriction
that $p\not\equiv 1\pmod{\ell}$ for all primes $p\mid n$.
We impose this restriction on the sums over $n$ in (\ref{E:1}) and
follow the proof in Section~\ref{S:counting}, except that we
multiply $f_\chi(n)$ and $f(n)$ by the multiplicative function $g(n)$
that outputs $1$ if the restriction is satisfied and $0$ otherwise.
In this way the right-hand side of (\ref{E:BIG}) still counts the relevant quantity.
That is, everything stays the same until we split into cases.

\underline{Case $\chi=\chi_0$}: Using
the Prime Number Theorem in Arithmetic Progressions, we find
$$\sum_{\substack{p\leq Y\\p\not\equiv 1\pmod{\ell}}}(\ell-1)\sim (\ell-2)Y/\log Y\,,$$
and Lemma~\ref{L:Wirsing} implies
$$
  F(Y)\sim \widehat{D}_\ell Y(\log Y)^{\ell-3}
  \,,\qquad
  \widehat{D}_\ell
  =
  \frac{1}{(\ell-3)!}
  \prod_{p}
  \widehat{D}_{\ell,p},
$$
where
$$
  \widehat{D}_{\ell,p}=
  \left(1-\frac{1}{p}\right)^{\ell-2}
  \begin{cases}
  1, & p\equiv 1\pmod{\ell}\,,\\
  1+\frac{\ell-1}{p}, & p\not\equiv 1\pmod{\ell}\,.
  \end{cases}
$$
Similarly, we have
$$
  F_{\chi}(Y)\sim c(\ell)\widehat{D}_\ell Y(\log Y)^{\ell-3}
  \,.
$$

\underline{Case $\chi\neq\chi_0$}:
We again
use $|F_\chi(Y)|\leq\sum_{n\leq Y}\mu^2(n)|f_\chi(n)|$
and apply Lemma~\ref{L:Wirsing}. 
As in Section~\ref{S:counting},
this leads to a bound of the form
$$
  F_\chi(Y)\ll \frac{Y}{\log Y}(\log Y)^\alpha
  \,.
$$
The following table gives the values of $f_\chi(p)$ for various congruence
classes modulo $\ell^2$.
\begin{center}
\begin{tabular}{l|l|l}
$f_\chi(p)$ & $\# n\in\left(\Z/\ell^2\Z\right)^\times$ & congruences\\
\hline
$0$ & $1$ & $p\equiv 1\pmod{\ell}$,\; $p^{\ell-1}\equiv 1\pmod{\ell^2}$\\
$0$ & $\ell-1$&$p\equiv 1\pmod{\ell}$,\; $p^{\ell-1}\not\equiv 1\pmod{\ell^2}$\\
%$0$ & $\ell$ & $p\equiv 0\pod{\ell}$\\
$\ell-1$ & $\ell-2$ & $p\not\equiv 1\pmod{\ell}$,\; $p^{\ell-1}\equiv 1\pmod{\ell^2}$\\
$-1$ & $(\ell-1)(\ell-2)$ & $p\not\equiv 1\pmod{\ell}$,\; $p^{\ell-1}\not\equiv 1\pmod{\ell^2}$
\end{tabular}
\end{center}
From the table, one reads off the value
$$
  \alpha=\frac{(\ell-1)(\ell-2)+(\ell-1)(\ell-2)}{\phi(\ell^2)}=
  2\left(1-\frac{2}{\ell}\right)
  \,,
$$
and therefore
$$
F_\chi(Y)\ll Y(\log Y)^{1-\frac{4}{\ell}}
\,.
$$

At this point, for clarity, we denote the new counting function by $\widehat{N}_\ell(X)$.
We have thus obtained $\widehat{N}_\ell(Y)\sim \widehat{C}_\ell Y(\log Y)^{\ell-3}$, where
one can write down the constant $\widehat{C}_\ell$ exactly as in (\ref{E:const}).
(The expression is the same but with $D_\ell$ replaced by $\widehat{D}_\ell$.)
This proves the theorem with a constant of proportionality
$\widehat{C}_\ell/C_\ell=(\ell-1)\widehat{D}_\ell/D_\ell=:1/A_\ell$.
The extra factor of $\ell-1$ arises from (\ref{L:logpower}).
\end{proof}

%%%%%%%%%%%%%%%%%%%%%%%%%%%%%%%%%%%%%%%%%%%%%%%%

\section{The average genus number}\label{S:average}

\begin{proof}[Proof of Theorem~\ref{T:average}]
The proof is similar to the one in the previous section.  In order to sum $g_K$ over $K\in\mathcal{F}_\ell$
with $|\Delta_K|\leq X$, we insert the multiplicative function $g(n)=\ell^{\widehat{\omega}(n)}$
immediately to the right of both sums over $n$ in (\ref{E:1}).  As before, the right-hand side of
(\ref{E:BIG}) still counts the relevant quantity, provided we multiply $f_\chi(n)$ and $f(n)$ by the function $g(n)$.
To be clear, the modified $f_\chi(n)$ is defined as
$$
  f_\chi(n)=\ell^{\widehat{\omega}(n)}
  \sum_{\substack{
  \rad(a)=n\\
  \text{$a$ is $\ell$-free}
  }}
    \chi(a)
    \,,
$$
and the left-hand side of (\ref{E:BIG})
becomes $(\ell-1)\widehat{N}_\ell(X)+1$,
where
$\widehat{N}_\ell(X)=\sum_{|\Delta_K|\leq X}g_K$.

\underline{Case $\chi=\chi_0$}:
Using the Prime Number Theorem in Arithmetic Progressions, we observe that
$$
  \sum_{p\leq Y}g(p)(\ell-1)
  =
  \hspace{-2ex}
  \sum_{\substack{p\leq Y\\p\equiv1\pmod{\ell}}}
    \hspace{-2ex}
  \ell(\ell-1)
  +
    \hspace{-2ex}
    \sum_{\substack{p\leq Y\\p\not\equiv 1\pmod{\ell}}}
      \hspace{-2ex}
    (\ell-1)
    \sim(2\ell-2)\frac{Y}{\log Y}
    \,.
$$
Therefore Lemma~\ref{L:Wirsing} gives
$$
  F(Y)\sim \widehat{D}_\ell Y(\log Y)^{2\ell-3}
  \,,\qquad
  \widehat{D}_\ell=
  \frac{1}{(2\ell-3)!}
  \prod_p
  \widehat{D}_{\ell,p},
 $$
 where 
$$
  \widehat{D}_{\ell,p}=
  \left(1-\frac{1}{p}\right)^{2\ell-2}
  \begin{cases}
  1+\frac{\ell(\ell-1)}{p},
  & p\equiv 1\pmod{\ell},\\
  1+\frac{\ell-1}{p},
  & p\not\equiv 1\pmod{\ell}.
  \end{cases}
$$
Similarly,
$$F_\chi(Y)\sim c(\ell)\widehat{D}_\ell Y(\log Y)^{2\ell-3}
\,.
$$

\underline{Case $\chi\neq \chi_0$}:
The bound works the same as in Section~\ref{S:counting},
but there are four cases to consider.
We have
$$
  f_\chi(n)=\ell^{\widehat{\omega}(n)}\begin{cases}
   \ell-1, & n^{\ell-1}\equiv 1\pmod{\ell^2}\,,\\
  -1, & n^{\ell-1}\not\equiv 1\pmod{\ell^2}\,.
  \end{cases}
$$
The cases (when $p\neq\ell$) are summarized in the following table:

\vspace{1ex}

\begin{center}

\begin{tabular}{l|l|l}
$f_\chi(p)$ & $\# n\in\left(\Z/\ell^2\Z\right)^\times$ & congruences\\
\hline
$\ell(\ell-1)$ & $1$ & $p\equiv 1\pmod{\ell}$,\; $p^{\ell-1}\equiv 1\pmod{\ell^2}$\\
$-\ell$ & $\ell-1$&$p\equiv 1\pmod{\ell}$,\; $p^{\ell-1}\not\equiv 1\pmod{\ell^2}$\\
$\ell-1$ & $\ell-2$ & $p\not\equiv 1\pmod{\ell}$,\; $p^{\ell-1}\equiv 1\pmod{\ell^2}$\\
$-1$ & $(\ell-1)(\ell-2)$ & $p\not\equiv 1\pmod{\ell}$,\; $p^{\ell-1}\not\equiv 1\pmod{\ell^2}$
\end{tabular}

\end{center}

\vspace{1ex}

Thus we have the bound
\begin{align*}
F_\chi(Y)
\ll
Y
(\log Y)^{\alpha-1}
\end{align*}
with
$$
  \alpha=\frac{\ell(\ell-1)\cdot 1+\ell\cdot(\ell-1)+(\ell-1)\cdot(\ell-2)+1\cdot(\ell-1)(\ell-2)}{\phi(\ell^2)}=4\left(1-\frac{1}{\ell}\right)
  \,.
$$
This completes the bound on the error term, which suffices since $3-4/\ell<2\ell-3$.

We have thus obtained $\widehat{N}_\ell(Y)\sim \widehat{C}_\ell Y(\log Y)^{2\ell-3}$, where one can write down the constant $\widehat{C}_\ell$ exactly as in (\ref{E:const}).
It follows that we recover the constant in the theorem as
$$\widehat{C}_\ell/C_\ell=(\ell-1)^{-\ell+1}\widehat{D}_\ell/D_\ell=:B_\ell\,.$$
\end{proof}

\section{Acknowledgement}

This research was conducted as part of the Research Experience for Undergraduates and Teachers
program
at California State University, Chico, supported by NSF grant DMS-2244020.

\bibliography{genus}{}

\begin{thebibliography}{10}

\bibitem{MR3814003}
K\"ubra Benli.
\newblock On the number of pure fields of prime degree.
\newblock {\em Colloq. Math.}, 153(1):39--50, 2018.

\bibitem{MR184933}
W.~E.~H. Berwick.
\newblock {\em Integral bases}, volume No. 22 of {\em Cambridge Tracts in Mathematics and Mathematical Physics}.
\newblock Stechert-Hafner, Inc., New York, 1964.

\bibitem{MR3090184}
Manjul Bhargava, Arul Shankar, and Jacob Tsimerman.
\newblock On the {D}avenport-{H}eilbronn theorems and second order terms.
\newblock {\em Invent. Math.}, 193(2):439--499, 2013.

\bibitem{MR4768704}
Manjul Bhargava, Takashi Taniguchi, and Frank Thorne.
\newblock Improved error estimates for the {D}avenport-{H}eilbronn theorems.
\newblock {\em Math. Ann.}, 389(4):3471--3512, 2024.

\bibitem{MR756082}
H.~Cohen and H.~W. Lenstra, Jr.
\newblock Heuristics on class groups of number fields.
\newblock In {\em Number theory, {N}oordwijkerhout 1983 ({N}oordwijkerhout, 1983)}, volume 1068 of {\em Lecture Notes in Math.}, pages 33--62. Springer, Berlin, 1984.

\bibitem{MR2745550}
Henri Cohen and Anna Morra.
\newblock Counting cubic extensions with given quadratic resolvent.
\newblock {\em J. Algebra}, 325:461--478, 2011.

\bibitem{MR491593}
H.~Davenport and H.~Heilbronn.
\newblock On the density of discriminants of cubic fields. {II}.
\newblock {\em Proc. Roy. Soc. London Ser. A}, 322(1551):405--420, 1971.

\bibitem{MR4598183}
Christopher Frei, Daniel Loughran, and Rachel Newton.
\newblock Distribution of genus numbers of abelian number fields.
\newblock {\em J. Lond. Math. Soc. (2)}, 107(6):2197--2217, 2023.

\bibitem{MR113868}
A.~Fr\"olich.
\newblock The genus field and genus group in finite number fields. {I}, {II}.
\newblock {\em Mathematika}, 6:40--46, 142--146, 1959.

\bibitem{MR435028}
Makoto Ishida.
\newblock {\em The genus fields of algebraic number fields}, volume Vol. 555 of {\em Lecture Notes in Mathematics}.
\newblock Springer-Verlag, Berlin-New York, 1976.

\bibitem{MR584554}
Makoto Ishida.
\newblock On the genus fields of pure number fields.
\newblock {\em Tokyo J. Math.}, 3(1):163--171, 1980.

\bibitem{MR625129}
Makoto Ishida.
\newblock On the genus fields of pure number fields. {II}.
\newblock {\em Tokyo J. Math.}, 4(1):213--220, 1981.

\bibitem{MR4048607}
Henry~H. Kim.
\newblock Genus numbers of cyclic and dihedral extensions of prime degree.
\newblock {\em Acta Arith.}, 192(3):289--300, 2020.

\bibitem{MR1878556}
Serge Lang.
\newblock {\em Algebra}, volume 211 of {\em Graduate Texts in Mathematics}.
\newblock Springer-Verlag, New York, third edition, 2002.

\bibitem{MR1884706}
Gunter Malle.
\newblock On the distribution of {G}alois groups.
\newblock {\em J. Number Theory}, 92(2):315--329, 2002.

\bibitem{MR2068887}
Gunter Malle.
\newblock On the distribution of {G}alois groups. {II}.
\newblock {\em Experiment. Math.}, 13(2):129--135, 2004.

\bibitem{MR1205901}
Daniel~C. Mayer.
\newblock Discriminants of metacyclic fields.
\newblock {\em Canad. Math. Bull.}, 36(1):103--107, 1993.

\bibitem{MR4649640}
Kevin~J. McGown, Frank Thorne, and Amanda Tucker.
\newblock Counting quintic fields with genus number one.
\newblock {\em Math. Res. Lett.}, 30(2):577--588, 2023.

\bibitem{MR4613609}
Kevin~J. McGown and Amanda Tucker.
\newblock Statistics of genus numbers of cubic fields.
\newblock {\em Ann. Inst. Fourier (Grenoble)}, 73(4):1365--1383, 2023.

\bibitem{MR2378655}
Hugh~L. Montgomery and Robert~C. Vaughan.
\newblock {\em Multiplicative number theory. {I}. {C}lassical theory}, volume~97 of {\em Cambridge Studies in Advanced Mathematics}.
\newblock Cambridge University Press, Cambridge, 2007.

\bibitem{MR3127806}
Takashi Taniguchi and Frank Thorne.
\newblock Secondary terms in counting functions for cubic fields.
\newblock {\em Duke Math. J.}, 162(13):2451--2508, 2013.

\bibitem{MR131389}
Eduard Wirsing.
\newblock Das asymptotische {V}erhalten von {S}ummen \"uber multiplikative {F}unktionen.
\newblock {\em Math. Ann.}, 143:75--102, 1961.

\end{thebibliography}
\bibliographystyle{plain}

\end{document}